\documentclass[graybox]{svmult}

\usepackage{type1cm}        % activate if the above 3 fonts are
\usepackage{makeidx}         % allows index generation
\usepackage{graphicx}        % standard LaTeX graphics tool
\usepackage{multicol}        % used for the two-column index
\usepackage[bottom]{footmisc}% places footnotes at page bottom
\usepackage{graphicx}  
\usepackage{caption}
\usepackage{subcaption}
\usepackage{newtxtext}       % 
\usepackage[varvw]{newtxmath} 
\usepackage{pinlabel}
\makeindex             % used for the subject index
\newcommand{\myjoin}[2]{\,\vphantom{\vee}_{#1}\!\!\vee_{#2}}

\newcommand{\bG}{\mathbb{G}}

\newcommand{\bH}{\mathbb{H}}

\newcommand{\ba}{\mathbin{\backslash}}
\newcommand{\con}{\mathbin{/}}

\theoremstyle{definition}
\newtheorem{terminology}{Terminology}
\newtheorem{construction}{Construction}

\begin{document}

\title*{A coarse Tutte polynomial for hypermaps}
\author{Joanna A. Ellis-Monaghan, \\ Iain Moffatt\\ Steven Noble}
\institute{Joanna A. Ellis-Monaghan\at Korteweg-de Vries Instituut voor Wiskunde, Universiteit van Amsterdam, Science Park 105-107, 1098 XG Amsterdam, The Netherlands, \email{j.a.ellismonaghan@uva.nl}
\and Iain Moffatt \at Department of Mathematics, Royal Holloway, University of London, Egham, TW20~0EX, United Kingdom, \email{iain.moffatt@rhul.ac.uk}
\and Steven Noble \at School of Computing and Mathematical Sciences,  Birkbeck, University of London, Malet Street, London, WC1E~7HX, United Kingdom, \email{s.noble@bbk.ac.uk}
}

\maketitle

%\linenumbers

\abstract{We give an analogue of the Tutte polynomial for hypermaps.  This polynomial can be defined as either a sum over subhypermaps, or recursively through deletion-contraction reductions where the terminal forms consist of isolated vertices. Our Tutte polynomial extends the classical Tutte polynomial of a graph as well as the Tutte polynomial of an embedded graph (i.e., the ribbon graph polynomial), and it is a specialization of the transition polynomial via a medial map transformation.  We
 give hypermap duality and partial duality identities for our polynomial, as well as
 some evaluations, and examine relations between our polynomial and other hypermap polynomials.
}

\section{Introduction}\label{sec:intro}

In this paper we introduce a Tutte polynomial for hypermaps as a direct generalisation of the Tutte polynomials for abstract graphs and for maps (graphs cellularly embedded on surfaces).

An edge of a graph may be defined as a multiset containing exactly two (not necessarily distinct) vertices. Hypergraphs generalize graphs by allowing hyperedges which are multisets containing any number of vertices.  A map may be thought of as a drawing of a graph on a surface (compact 2-manifold) so that the edges do not cross and so that each face is a region of the surface homeomorphic to a disc. Similarly, a hypermap may be thought of as a hypergraph drawn on a surface. (See Section~\ref{sec:hyper} for formal definitions.)

Hypergraph and hypermap models are attracting increasing attention as applications of traditional network models seek refinements through higher-order interactions, with~\cite{Bat21,Bat20} providing particularly compelling overviews of the urgency for research in this direction.  These higher-order interactions, which consider connections among multiple nodes instead of just pairwise connections,  correspond to systems with hyperedges in place of simply edges. 
Applications, particularly in physics, have already led to extensions of the Potts model to hypergraphs, for example~\cite{BB09,Grim94}, and efforts have begun to extend the Tutte polynomial to hypergraphs, for example in \cite{BKP22, Kal13}.

Our interest here is in constructing a Tutte polynomial for hypermaps.
To do this, since hypermaps generalize maps, it is natural to build on the existing theory of map polynomials.
There is a rich literature on analogues of the Tutte polynomial for maps (see, e.g.,~\cite{zbMATH01801590,zbMATH07635160,zbMATH07395405,zbMATH06832600,iainmatrixannals} and the survey~\cite{Chm22}).
By adapting the approach described in~\cite{iainmatrixannals}, we begin by constructing a version of the dichromatic polynomial for hypermaps. Classical connections between the dichromatic and Tutte polynomials then lead us to a Tutte polynomial for hypermaps. 
The reason we take this approach is that it keeps deletion-contraction properties at its heart, and so the polynomials we construct can be defined by recursive deletion-contraction relations with a base case consisting of hyperedgeless hypermaps. 

The hypermap deletion and contraction relations we use here are direct extensions of map operations.
Deletion of a hyperedge removes the entire edge from the hypermap, and contraction takes the partial dual of a hyperedge (see~\cite{CV22,ben}) and then removes it.  Since there are other, more refined, definitions of hyperedge deletion and contraction, such as those given in~\cite{CHpreHyperDelCon,CHpreWhit}, we describe our hypermap Tutte polynomial as `coarse' because we use these `whole hyperedge' operations. The coarse Tutte polynomial for hypermaps that emerges from this choice captures many desirable properties,  and we give several identities and evaluations for it.

We also discuss various interconnections between our hypermap polynomial and other polynomials from the literature.  We see that for maps, which are equivalent to ribbon graphs, the coarse Tutte polynomial coincides with the ribbon graph polynomial.  We also show that, via a medial map construction, the coarse Tutte polynomial for hypermaps is a specialization of the transition polynomial.  Finally, we make a comparison with the recent hypermap Whitney polynomial of Cori and Hetyei~\cite{CHpreWhit}, which uses a more refined edge deletion, showing that neither of the two polynomials determine the other.

J\'anos Makowsky is a keen collector of graph polynomials for his zoo~\cite{zbMATH05552168}. Hypermaps can provide many more such specimens, which we believe will enrich his zoo. Potentially they provide a fertile setting for the extension of many of his ideas, particularly his Monadic Second Order Logic framework for graph polynomials (see ~\cite{zbMATH02136947} and also ~\cite{MR3960189}).  
We close by suggesting this extension as a possible direction for future research and discussing complexity issues.

\section{Hypermaps}\label{sec:hyper}
We allow graphs to have loops and parallel edges, and follow the terminology in~\cite{zbMATH05202336}.

An \emph{embedded graph} or \emph{map} consists of a closed surface $\Sigma$ (not necessarily connected and possibly nonorientable), a set of distinct points on the surface (called \emph{vertices}) and a set of simple arcs (called \emph{edges}) whose ends lie on vertices. Furthermore, an edge may only intersect a vertex at its ends, and edges may not intersect except at their ends.
The vertices and edges divide the surface into regions, called \emph{faces}, and we insist that each face is homeomorphic to a disc. (Thus we only consider \emph{cellularly embedded graphs} here.)
A consequence of this is that each component of the underlying graph must lie in a different connected component of the closed surface.
In this paper we use both the terms embedded graph and map. We favour the term map but use embedded graph when it is the preferred term in the sources we are citing. (We also work with ribbon graphs in Section~\ref{ss:classical} for this reason.)

We shall define a \emph{hypergraph} to be a bipartite graph $G=(V_v\sqcup V_e ,E)$ in which multiple edges are allowed but no vertex of $V_e$ is isolated. The set $V_v$ forms the set of \emph{hypervertices} of the hypergraph, and each vertex in  $V_e$ together with its incident edges forms a \emph{hyperedge}. 
A \emph{hypermap} is an embedded hypergraph.

Figure~\ref{f.ex1a} shows a hypermap with $4$ hypervertices (the black vertices) and $3$ hyperedges (the white vertices) embedded in the sphere. Here $V_v=\{v_1,v_2,v_3,v_4\}$, 
 $V_e=\{e_1,e_2,e_3\}$ and the faces are labelled $f_1, \ldots , f_4$. 

\begin{figure}[ht!]
     \centering
  %%%%%%%%%%%%%
  %%%%%2222%%%%%%%%
        \hfill
        \begin{subfigure}[c]{0.45\textwidth}
        \centering
         \labellist
        \small\hair 2pt
        \pinlabel $v_1$ at 23 127
        \pinlabel $v_2$ at 129 126
        \pinlabel $v_3$ at 160 83
        \pinlabel $v_4$ at 77 30
        \pinlabel $e_1$ at 39 89
         \pinlabel $e_2$ at 75 167
          \pinlabel $e_3$ at 127 70
          \pinlabel $f_1$ at 70 130
        \pinlabel $f_3$ at 66 66
        \pinlabel $f_2$ at 95 87
        \pinlabel $f_4$ at 125 44
        \pinlabel $S^2$ at 120 173
        \endlabellist
\includegraphics[scale=0.8]{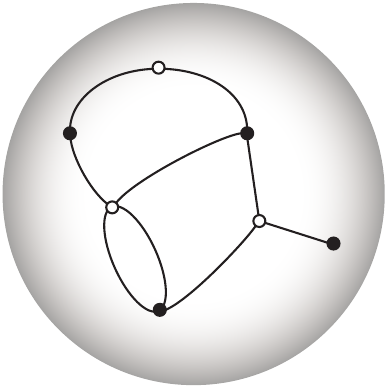}
        \caption{A hypermap in the sphere.}
        \label{f.ex1a}
     \end{subfigure}
%%%%
  %%%%%1111%%%%%%%%
        \hfill
        \begin{subfigure}[c]{0.45\textwidth}
        \centering
         \labellist
        \small\hair 2pt
        \pinlabel $v_1$ at  15 100
        \pinlabel $v_2$ at  101 100
        \pinlabel $v_3$ at  142 47
       \pinlabel $v_4$ at  58 15
       \pinlabel $e_2$ at  54 132
         \pinlabel $e_1$ at  37 65
          \pinlabel $e_3$ at  107 57
          \pinlabel $f_1$ at  53 106
        \pinlabel $f_3$ at  42 38
        \pinlabel $f_2$ at  79 62
        \pinlabel $f_4$ at  115 20
      \endlabellist
      \includegraphics[scale=0.9]{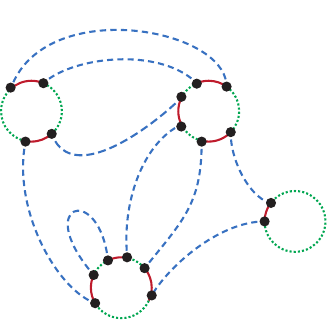}
      \vspace{5mm}
        \caption{The corresponding  gehm.}
        \label{f.ex1b}
     \end{subfigure}

      \begin{subfigure}[c]{0.45\textwidth}
        \centering
         \labellist
        \small\hair 2pt
        \pinlabel $S^2$ at 165 130 
      \endlabellist
      \includegraphics[scale=0.9]{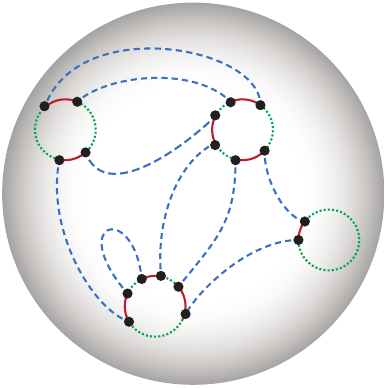}
        \caption{The natural embedding of the  gehm in the sphere. }
        \label{f.ex1c}
     \end{subfigure}
\caption{A hypermap, its gehm and the natural embedding of the gehm.}
\label{f.ex1}
\end{figure}

It is convenient to describe a hypermap as a \emph{graph encoded hypermap}, which we abbreviate as \emph{gehm}.  A gehm is a properly edge 3-coloured cubic graph in which edges are coloured from the set $\{b,g,r\}$ (standing for $\{\text{blue},\text{green},\text{red}\}$). In addition, we allow our gehms to have $g$-coloured edges which do not meet any vertices. We call these \emph{isolates}. They appear in the gehm as isolated $g$-coloured circles.
 If each $b$--$r$-cycle in a gehm has length exactly four then we say it is a \emph{graph encoded map} or \emph{gem} (and it then describes an embedded graph~\cite{zbMATH00824939,zbMATH03728291}).
Gehms describe hypermaps and vice versa, but we must develop some terminology before giving this correspondence. Figures~\ref{f.ex1a} and~\ref{f.ex1b} show, respectively, a hypermap and a gehm which we will see later correspond to each other. The labels on the gehm indicate the correspondence and may be ignored for the moment. As figures can be misleading, we emphasise that gehms are (abstract) graphs and are not embedded.

For readability in both colour and black and white printing, we adopt the following convention in our figures. We use blue dashed lines to denote $b$-edges, green dotted lines to denote $g$-edges, and red solid lines to denote $r$-edges.

We set up the following terminology, trusting that the rationale for the names will become clear.
\begin{terminology}\label{term1}
~
\begin{itemize}
\item A \emph{gehm-edge} is an edge of the cubic graph that forms the gehm.
\item A  \emph{gehm-vertex} is a vertex of the cubic graph that forms the gehm.
\item A $b$--$r$-cycle in the gehm is called a  \emph{hyperedge}: it represents an edge of the hypermap. 
\item A $b$--$g$-cycle in the gehm is called a \emph{hyperface}: it represents a face of the hypermap. 
\item A $g$--$r$-cycle in the gehm is called a \emph{hypervertex}: it represents a vertex of the hypermap.
\item An \emph{isolate} is both a \emph{hyperface} and a \emph{hypervertex}: it represents a component of a hypermap that consists of an isolated vertex embedded in the sphere.
\item $E(\bH)$, $V(\bH)$ and $F(\bH)$ are, respectively, the sets of hyperedges, hypervertices and hyperfaces of the gehm $\bH$, and $e(\bH)=|E(\bH)|$, $v(\bH)=|V(\bH)|$ and $f(\bH)=|F(\bH)|$. Note that both $V(\bH)$ and $F(\bH)$ include all the isolates.
\item $k(\bH)$ denotes the number of components of the gehm $\bH$: each represents a component of the hypermap. Isolates contribute to $k(\bH)$.
\item The \emph{degree} $d(e)$ of a hyperedge $e$ is half the number of edges in its $b$--$r$ cycle, and this coincides with the degree of the corresponding vertex in the bipartite graph $G=(V_v\sqcup V_e ,E)$ whose embedding gives the hypermap.  We write $d(\bH)=\sum_{e\in E(\bH)} d(e)$. Note that $d(\bH)$ is equal to the number of $r$-edges, the number of $b$-edges and the number of $g$-edges after excluding isolates. 
\item  A gehm is \emph{orientable} if it is bipartite, in which case an \emph{orientation} is a choice of vertex 2-colouring. If a gehm is not bipartite then it is \emph{non-orientable}. 
\item The \emph{Euler genus}, $\gamma(\bH)$, is defined through \emph{Euler's formula}  
\[
\gamma(\bH)=2k(\bH)-v(\bH)-e(\bH)+d(\bH)-f(\bH).\] 
\item The \emph{genus} of $\bH$ is $\gamma(\bH)$ when  $\bH$ is non-orientable, and is $\gamma(\bH)/2$ when $\bH$ is orientable.
\end{itemize}
\end{terminology}

For example, consider the gehm in Figure~\ref{fexgehm}. It has nine gehm-edges and six   gehm-vertices which are labelled $1, \ldots ,6$. It has two hypervertices given by the $g$--$r$-cycles $1\,2\,1$ and $3\,4\,5\,6 \,3$; one hyperedge of degree 3 given by the $b$--$r$-cycle $1\,6\,3\,4\,5\,2\,1$; and two hyperfaces given by the $b$--$g$-cycles $3\,4\,3$ and $1\,6\,5\,2\,1$. The gehm is orientable and has genus zero.

\begin{figure}[ht!]
\centering   
   \labellist
        \small\hair 2pt
        \pinlabel $1$ at 27 66  
        \pinlabel $2$ at   27 26
        \pinlabel $3$ at  117 65
       \pinlabel $4$ at   117 32
       \pinlabel $5$ at   148 39
         \pinlabel $6$ at   148 58
      \endlabellist
      \includegraphics[scale=1]{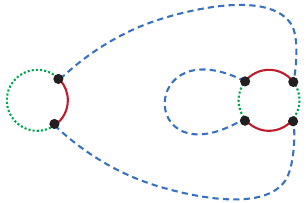}
\caption{An example of a gehm.}
\label{fexgehm}
\end{figure}

The gehm in Figure~\ref{f.ex1b} has four hypervertices, three hyperedges and four hyperfaces. The degrees of its hyperedges are $2$, $3$ and $4$. It has Euler genus $0$ and is orientable.
The labels $e_i$, $f_i$, and $v_i$ indicate the coloured cycles corresponding to the hyperedges, hyperfaces and hypervertices. 
We emphasize the gehm is not embedded, so the labels refer to cycles in the gehm not faces in the drawing on the page.

Two gehms are \emph{equivalent} if there is an isomorphism between them that preserves the edge-colouring. If the gehms are oriented the isomorphism should also preserve the vertex 2-colouring. 

\medskip
Each gehm has a natural embedding in a surface, as follows. 
\begin{construction} \label{makehypermap} 
We construct a complex from a gehm.  Ignore isolates for the moment.  Consider the gehm as a 1-complex with vertices giving the 0-simplices and edges the 1-simplices (following the standard construction for graphs).
Take one disc for each hypervertex ($g$--$r$-cycle) and identify the boundary of this disc with the hypervertex. Do the same for each hyperedge and hyperface.

Finally, consider each isolate as a copy of $S^1$ and embed each in a distinct sphere. (We think of the two hemispheres as corresponding to a vertex and a face.) 
Thus we have obtained a cellular embedding of the gehm in a surface. 
We call this the \emph{natural embedding} of a gehm. The 2-cells of this natural embedding correspond to hypervertices, hyperedges and hyperfaces of the gehm.
\end{construction}

The natural embedding of the gehm in Figure~\ref{f.ex1b} is shown in Figure~\ref{f.ex1c}. 

From the natural embedding of a gehm we can obtain a hypermap in the obvious way by placing one vertex in each hypervertex-disc to get $V_v$,  and one vertex in each hyperedge-disc to get $V_e$. For each intersection between a hypervertex-disc and a hyperedge-disc embed an edge between the corresponding vertices in the usual way (the edges should not intersect themselves or each other). For an isolate place one vertex of $V_v$ in the sphere.
This is clearly reversible and when combined with Construction~\ref{makehypermap} gives a correspondence between hypermaps and gehms. 
Note that this correspondence draws the natural embedding of the gehm and its corresponding hypermap in the same surface.  Thus, a hypermap and the natural embedding of its corresponding gehm are always in homeomorphic surfaces.

All of the parameters and terminology given in Terminology~\ref{term1} align with their  standard hypermap usage. The only terms that perhaps require some comment are Euler genus  and orientability.   
Let $\bH$ be a gehm corresponding to a hypermap $G$ given by the bipartite graph $(V_v\sqcup V_e ,E)$  embedded in a surface $\Sigma$. 
The Euler genus $\gamma(G)$ of $G$ is the Euler genus of $\Sigma$.  The Euler genus of a disconnected surface is the sum of the Euler genera of its components. We can disregard isolates as they do not contribute to the Euler genus of either $\bH$ or $G$. We see that the definition of $\gamma(\bH)$ is consistent with $\gamma(G)$ as follows.  By Euler's formula, $\gamma(G) = 2k(G)-(|V_v| + |V_e|) + e(G) -f(G) $.  However, $|V_v| =v(\bH)$ and $|V_e| =e(\bH)$, while $k(G) = k(\bH)$ and $f(G) = f(\bH)$, and finally $e(G) = d(\bH)$.  With this, $\gamma(\bH) = \gamma(G)$, and this common value is also the Euler genus of the surface created in constructing the natural embedding of the gehm.

For orientability, let $\bH$ be a gehm, $G$ be the corresponding hypermap, and $\bG$ this hypermap described as a gem. (Recall that $\bG$ is an embedded bipartite graph, so may be described by a gem, that is, a gehm in which $d(e)=2$ for every hyperedge $e$.)
By considering how $\bG$ can be obtained directly from $\bH$ it is clear that $\bH$ is bipartite if and only if $\bG$ is. Then by a standard result about gems (see e.g.,~\cite[Theorem~4.3]{zbMATH00824939}) $\bG$ is bipartite if and only if $G$ is orientable and it follows that $\bH$ is orientable if and only if $G$ is.

It follows from the correspondence and standard properties of the Euler genus of a surface that $\gamma(\bH)\geq 0$ and if $\bH$ is orientable then $\gamma(\bH)$ is even.

\section{A Tutte polynomial for hypermaps}

\subsection{Duality and minors}

There are six ways to permute the colours of the edges of a gehm, each of which corresponds to a natural duality or triality operation.
\begin{definition}\label{def:duality}
Let $\bH$ be a gehm, and $\mu$ be a permutation on the set $\{b,g,r\}$. 
Then we use  $\bH^{\mu}$ to denote the gehm obtained from $\bH$ by, for each $c\in\{b,g,r\}$ changing all $c$-coloured edges to $\mu(c)$ coloured edges.  
\end{definition}

If $\mu$ is of order 2, then $\bH^{\mu}$ is said to be a \emph{dual}. In particular, $\bH^{(br)}$ is the usual geometric dual, denoted by $\bH^*$, which interchanges the faces and vertices of a  hypermap. 
The gehm $\bH^{(bgr)}$ is the \emph{trial} of $\bH$, introduced by Tutte in~\cite{Tu73}.
In this paper we focus on geometric duals: a fuller study of duality (and minors) can be found in~\cite{referencenotfound}.

\begin{definition}\label{def:pardual}
Let $\bH$ be a gehm, and let $e$ be a hyperedge (i.e., a $b$--$r$-cycle).  The \emph{partial dual}   $\bH^e$ is the gehm obtained by interchanging the colours of the gehm-edges in the cycle $e$.
\end{definition}

Figures~\ref{f.ex2a} and~\ref{f.ex2b} show, respectively, the dual and the partial dual with respect to the degree four hyperedge of the gehm from Figure~\ref{f.ex20cex2}. 

Notice that in moving from Figure~\ref{f.ex20cex2} to Figure~\ref{f.ex2a}, the colours $b$ and $r$ are swapped, so every $b$--$g$ cycle becomes a $g$--$r$ cycle and vice versa. Thus every hyperface becomes a hypervertex and vice versa. 
While $\bH$ and $\bH^*$ in the figure are of genus zero, the partial dual $\bH^e$ is of genus two. This can be verified by using Euler's formula.

\begin{figure}[ht!]
     \centering
  %%%%%%%%%%%%%
  %%%%%2222%%%%%%%%
    
        \begin{subfigure}[c]{0.45\textwidth}
        \centering
    \includegraphics[scale=0.9]{f2}
        \caption{The gehm $\bH $.}
        \label{f.ex20cex2}
     \end{subfigure}
  %%%%%1111%%%%%%%%

        \begin{subfigure}[c]{0.45\textwidth}
        \centering
    \includegraphics[scale=0.9]{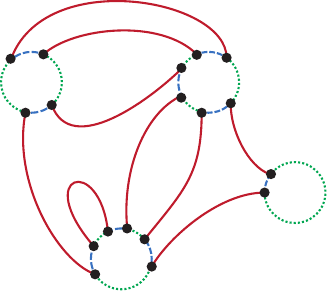}
        \caption{The dual $\bH^* $.}
        \label{f.ex2a}
     \end{subfigure}
  %%%%%1111%%%%%%%%
    \hfill
        \begin{subfigure}[c]{0.45\textwidth}
        \centering
\includegraphics[scale=0.9]{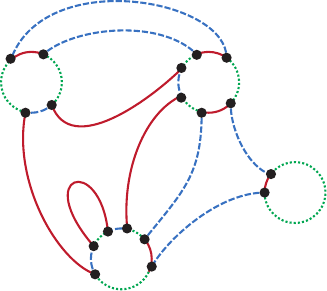}
        \caption{The partial dual $\bH^e $ where $e$ is the degree four edge.}
        \label{f.ex2b}
     \end{subfigure}
\caption{The dual and a partial dual of a gehm $\bH$. }
%in Figure~\ref{f.ex1b}.}
\label{f.ex2}
\end{figure}

For distinct hyperedges $e$ and $f$, it is straightforward to check that $(\bH^e)^f=(\bH^f)^e$. Thus we may unambiguously extend the definition of partial duality to sets of hyperedges. For a set $A$ of hyperedges, the partial dual $\bH^A$ is defined to be the result of computing the partial dual with respect to each edge of $A$ in any order.

Partial duals for hypermaps were introduced independently in~\cite{CV22,ben}. We only consider partial duals of hyperedges here, however, as in~\cite{CV22}, this definition is easily extended to hypervertex and hyperface partial duals.

\medskip

We next consider operations of deletion and contraction.
Because a hyperedge can in general be incident with many hypervertices, there are many possible definitions of deletion.  Below, we use one of the coarsest possible  definitions, and remove the entire hyperedge  without removing its incident hypervertices.    This is sometimes called `weak hyperedge deletion', in contrast to `strong hyperedge deletion', which deletes the incident hypervertices   as well.  There are also other hyperedge deletion models.  For example, in~\cite{CHpreHyperDelCon} 
the cyclic order of the hypervertices about the hyperedge $e$ in the embedding is used
to define a non-crossing partition of the hypervertices incident with $e$ and replace $e$ with multiple smaller hyperedges corresponding to the blocks of the partition. We shall discuss this further in  Subsection~\ref{CHpoly}. 

A common method of contracting the hyperedges of a hypergraph  is often described as identifying a hyperedge and its incident hypervertices to form a new hypervertex. This agrees with our definition of hyperedge contraction in hypermaps below provided that all the hypervertices are distinct.  If a hyperedge in a hypermap is incident with the same hypervertex multiple times, this does not hold, and multiple hypervertices may be created.  (This is also  what happens when contracting an orientable loop in a ribbon graph~\cite{zbMATH06151556}.)

Let $v$ be a vertex of degree two in a graph. By \emph{suppressing} $v$, we mean the following operation. If the only edge incident with $v$ is a loop, then replace $v$ and its incident edge with an isolated edge not adjacent to any vertex (in what follows this edge will be an isolate). Otherwise contract one of the edges incident with $v$.

\begin{definition}\label{def:delcon}
Let $\bH$ be a gehm, and let $e$ be a hyperedge (i.e., a $b$--$r$-cycle) then 
\begin{enumerate}
    \item
$\bH$ \emph{delete} $e$, denoted by $\bH\ba e$ is the gehm obtained from $\bH$ by deleting the $b$-gehm-edges in the $b$--$r$-cycle $e$, contracting the $r$-gehm-edges in the $b$--$r$-cycle $e$ and then suppressing the resulting vertices of degree two;
    \item
$\bH$ \emph{contract} $e$, denoted by $\bH \con e$ is the gehm obtained from $\bH$ by deleting the $r$-gehm-edges in the $b$--$r$-cycle $e$, contracting the $b$-gehm-edges in the $b$--$r$-cycle $e$ and then suppressing the resulting vertices of degree two.
\end{enumerate}
\end{definition}

Figure~\ref{f.ex3a} shows the effect of deleting the edge with degree three from the gehm in Figure~\ref{f.ex3newa};
Figure~\ref{f.ex3b} shows the effect of contracting the edge with degree four from the gehm in Figure~\ref{f.ex3newa}. Notice that in this example, both deletion and contraction create an additional component, here an isolate.  Also, while the diagrams in Figures~\ref{f.ex3newa} and~\ref{f.ex3a} coincide with natural embeddings of the gehm in one sphere (Figure~\ref{f.ex3newa}) or two spheres (Figure~\ref{f.ex3a}), the diagram in Figure~\ref{f.ex3b} does not. In a natural embedding, each isolate is in a separate spherical component.

\begin{figure}[ht!]
     \centering
  %%%%%%%%%%%%%
   \begin{subfigure}[c]{0.45\textwidth}
        \centering
 \includegraphics[scale=0.9]{f2}
        \caption{The gehm $\bH $.}
        \label{f.ex3newa}
     \end{subfigure}
  %%%%%2222%%%%%%%%
  
        \begin{subfigure}[c]{0.45\textwidth}
        \centering
\centering
\includegraphics[scale=0.9]{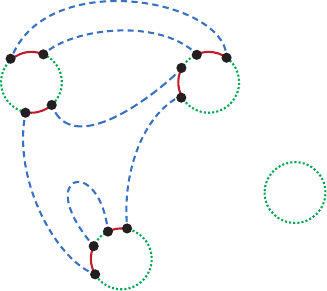}
        \caption{$\bH\ba e$ where $e$ is the degree three hyperedge.}
        \label{f.ex3a}
     \end{subfigure}
%%%%
\hfill
  %%%%%1111%%%%%%%%
\begin{subfigure}[c]{0.45\textwidth}
        \centering
  \includegraphics[scale=0.9]{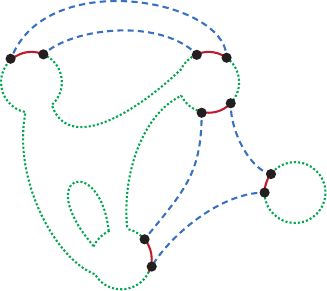}        \caption{$\bH\con f$ where $f$ is the degree four hyperedge.}
        \label{f.ex3b}
     \end{subfigure}
\caption{Deleting and contracting a hyperedge in a gehm $\bH$.}
\label{f.ex3}
\end{figure}

The following lemma is straightforward.
\begin{lemma}\label{lem:delconeasy}
Let $\bH$ be a gehm, and let $e$ and $f$ be distinct hyperedges. Then the following hold.
\begin{enumerate}
    \item $\bH^e\ba e=\bH\con e$;
    \item $(\bH \ba e) \ba f = (\bH \ba f) \ba e$;
    \item $(\bH \con e) \con f = (\bH \con f) \con e$;
    \item $(\bH \con e) \ba f = (\bH \ba f) \con e$.
    \item If $\bH$ is orientable, then so are both $\bH\ba e$ and $\bH\con e$.
\end{enumerate}
\end{lemma}
It follows that we may carry out a sequence of deletion and contraction operations in any order without affecting the result. In particular, for a set $A$ of hyperedges we may unambiguously define $\bH\ba A$ to be the result of deleting all the edges in $A$. We let $\bH_{|A}= \bH \ba (E(\bH)-A)$, $v(A)=v(\bH_{|A})=v(\bH)$, $k(A)=k(\bH_{|A})$, $e(A)=e(\bH_{|A})=|A|$,  
$f(A)=f(\bH_{|A})$ and $\gamma(A)=\gamma(\bH_{|A})$.
Finally, we let $d(A)=d(\bH_{|A})=\sum_{e\in A}d(e)$.  

\medskip

Hyperedge deletion for gehms and hypermaps correspond with each other.  A description of deletion for maps can be found in, for example,~\cite{zbMATH07395405,iainmatrixannals}. It acts as follows.
If $G$ is a map with an edge $e$ then $G\ba e$ is obtained by removing the edge $e$ from the map together with its adjacent  face or faces. This gives a surface with one boundary component. Next cap off the hole by identifying its boundary with the boundary of a disc, resulting in a map. (For readers familiar with ribbon graphs, this corresponds exactly to deleting an edge of a ribbon graph.)

Let $\bH$ be a gehm and $G$ be its corresponding hypermap (i.e., embedded bipartite graph). Suppose $e$ is a hyperedge of   $\bH$ and $v_e$ its corresponding vertex in $G$. Then the hypermap corresponding to $\bH\ba e$ is obtained by deleting all the edges incident with $v_e$ then deleting $v_e$ including the sphere it is embedded in.

\subsection{Defining the polynomials}

In this section we introduce a Tutte polynomial for hypermaps. This polynomial is a direct generalisation of the well-studied Tutte polynomial for maps (which is also known as the ribbon graph polynomial)~\cite{zbMATH01801590,zbMATH07635160,zbMATH07395405,zbMATH06832600,iainmatrixannals}. 
Although our polynomial is naturally defined in terms of deletion-contraction relations, it is more convenient to begin with an analogue of the dichromatic polynomial.  This immediately generalizes to a multivariate version that facilitates partial duality identities which lead in turn to full duality formulas.  A standard argument, albeit using hypermap properties, shows that the hypermap dichromatic polynomial has a deletion-contraction reduction. Our analogue of the dichromatic polynomial then leads to a hypermap analogue of the Tutte polynomial.

\subsubsection{A hypergraph dichromatic polynomial}

\begin{definition}\label{def:ovdichrom}
The \emph{dichromatic polynomial} $Z(\bH;u,v)$ of a gehm $\bH$ is defined as follows:
\[ Z(\bH;u,v)= \sum_{A\subseteq E(\bH)} u^{d(A)-|A|} v^{f(A)}.\]
\end{definition}

It is straightforward to extend the dichromatic polynomial to a multivariate version, which will enable us to easily establish a duality relation. In the multivariate polynomial, the variable $u$ is replaced by a family of commuting variables $\mathbf u:=\{u_e\}_{e\in E(\bH)}$ indexed by the hyperedges of $\bH$.

\begin{definition}\label{def:mvdichrom}
The \emph{multivariate dichromatic polynomial} $Z(\bH;\mathbf u,v)$ of a gehm $\bH$ is defined as follows:
\[ Z(\bH;\mathbf u,v)= \sum_{A\subseteq E(\bH)} \Big(\prod_{e\in A}u_e^{d(e)-1}\Big)\; v^{f(A)}.\]
\end{definition}

The multivariate dichromatic polynomial satisfies the following recurrence relation. For this we note that a gehm $\bH$ with no hyperedges comprises $v(\bH)=f(\bH)$ isolates. 
\begin{lemma}\label{lem:multidichromdelcon}
Let $\bH$ be a gehm. Then $Z(\bH;\mathbf u,v) = v^{f(\bH)}$ if $\bH$ has no hyperedges and otherwise for each hyperedge $e$,
\[ Z(\bH;\mathbf u,v) = 
Z(\bH\ba e;\{u_e\}_{e\in E(\bH\ba e)},v) + u_e^{d(e)-1}Z(\bH\con e;\{u_e\}_{e\in E(\bH\con e)},v) .\]
\end{lemma}
\begin{proof}
The argument is very standard so we spare the reader the details beyond noting the key observation that for every subset $A$ of $E(\bH)-\{e\}$, we have $f(A\cup \{e\})= f((\bH \con e)_{|A})$.    
\end{proof}

From this we immediately deduce the following.
\begin{corollary}\label{cor:dichromdelcon}
Let $\bH$ be a gehm. 
Then $Z(\bH;u,v) = v^{f(\bH)}$ if $\bH$ has no hyperedges and otherwise for each hyperedge $e$,
\[Z(\bH;u,v) = 
Z(\bH\ba e;u,v) + u^{d(e)-1} Z(\bH\con e;u,v).\]
\end{corollary}

\subsubsection{Duality}

We first examine the effect of partial duality on $Z(\bH;\mathbf u,v)$.
Consider a gehm $\bH$ and a subset $X$ of its hyperedges. 
We  identify the edges of $\bH$ and $\bH^X$ in the natural way.
Given $\mathbf u = \{u_e\}_{e\in E(\bH)}$, we define $\mathbf u^X=\{u'_e\}_{e\in E(\bH^X)}$ by
\[  u'_e = \begin{cases} 1/u_e &\text{if $e\in X$,} \\ u_e& \text{if $e\notin X$.}\end{cases} \]

\begin{proposition}
Let $\bH$ be a gehm, and let $X$ be a subset of its hyperedges. Then
\[ Z(\bH;\mathbf u,v) = \Big(\prod_{e\in X} u^{d(e)-1}_e\Big) \, Z(\bH^X;\mathbf u^X,v).\]
\end{proposition}

\begin{proof}
For each hyperedge $e$ and subset $A$ of hyperedges of $\bH$, we have  $f(\bH_{|A})=f((\bH^e)_{|A\bigtriangleup \{e\}})$, so
\begin{align*} Z(\bH;\mathbf u,v) &= \sum_{A\subseteq E(\bH)-\{e\}} \Big(\prod_{h\in A}u^{d(h)-1}_h\Big)
( v^{f(A)} + u_e^{d(e)-1}v^{f(A\cup \{e\})})\\
&= \sum_{A\subseteq E(\bH^e)-\{e\}} \Big(\prod_{h\in A}u^{d(h)-1}_h\Big)
( v^{f((\bH^e)_{|A\cup \{e\}})} + u^{d(e)-1}_ev^{f((\bH^e)_{|A})})\\
&=u^{d(e)-1}_e \sum_{A\subseteq E(\bH^e)-\{e\}} \Big(\prod_{h\in A}u^{d(h)-1}_h\Big)
\Big( \frac{1}{u^{d(e)-1}_e} v^{f((\bH^e)_{|A\cup \{e\}})} + v^{f((\bH^e)_{|A})}\Big)
\\&= u^{d(e)-1}_e \,Z(\bH^e;\mathbf u^{\{e\}},v).\end{align*}
The result now follows by induction on $|X|$.
\end{proof}

From this we deduce the following.
\begin{corollary}\label{cor:dichromdual}
Let $\bH$ be a gehm. Then
\[ Z(\bH; u,v) = u^{d(\bH)-e(\bH)} \, Z(\bH^*;1/u,v). \]
\end{corollary}

\subsubsection{Translating to the associated hypergraph Tutte polynomial}

We wish to define an analogue of the Tutte polynomial for hypermaps. For this we take the approach described in~\cite{iainmatrixannals} where it is explained how   the classical connection between the dichromatic and Tutte polynomials of a graph (see, e.g.,~\cite{handbookch2}) can be used to derive a Tutte polynomial of maps. 
The three key properties that our Tutte polynomial for hypermaps, $T(\bH;x,y)$, should satisfy are: (1) it should be equivalent (up to change of variables and multiplication by simple prefactors) to the dichromatic polynomial; (2) it should satisfy the duality relation $T(\bH;x,y)=T(\bH^*;y,x)$; and (3) $T(\bH;x,y)$ should coincide with the Tutte polynomial of a map~\cite{zbMATH01801590,zbMATH07635160,zbMATH07395405,zbMATH06832600,iainmatrixannals} when $\bH$ is a gem, that is a hypermap in which each hyperedge has degree two, and hence represents a map.
For this we define
\[
\rho(\bH)= v(\bH)-k(\bH)+\tfrac{1}{2}\gamma(\bH),
\]
and
for a set $A$ of hyperedges of $\bH$ we let 
\begin{align*}
\rho(A)=\rho(\bH_{|A})&= v(A)-k(A)+\tfrac{1}{2}\gamma(A)
\\&=
\frac 12 (v(A)+d(A)-|A|-f(A)),
\end{align*}
where the last equality follows by Euler's Formula.
It follows immediately from properties of $\gamma$ that we have established that $\rho(A)\geq 0$, and if $\bH$ is orientable then $\rho(A)$ is integral for each subset of its edges (since $\gamma(A)$ must be even).

\bigskip

We can now introduce our  Tutte polynomial for hypermaps as a sum over sets of hyperedges. The reader will likely notice a striking similarity with the definition of the classical Tutte polynomial of a graph. As we shall see, this similarity is an important feature of the polynomial.
\begin{definition}\label{def:hmtutte} For a gehm $\bH$, we define its \emph{Tutte polynomial} by
\begin{equation} \label{eq:Tuttedef} T(\bH;x,y) = \sum_{A\subseteq E(\bH)} (x-1)^{\rho(\bH)-\rho(A)} (y-1)^{d(A)-|A|-\rho(A)}.\end{equation}
\end{definition}

\begin{proposition}
The Tutte polynomial is a polynomial in $\sqrt{x-1}$ and $\sqrt{y-1}$.
\end{proposition}
\begin{proof}
Consider the hypermap (i.e., the embedded bipartite graph $G$) corresponding to the gehm $\bH$. As deleting edges in a map cannot increase genus or decrease the number of components, $ v(\bH)-k(\bH) \geq v(A)-k(A)$ and  $\gamma(\bH) \geq \gamma(A) $. Thus $\rho(\bH)-\rho(A)\geq 0$ and the $(x-1)$ exponent is non-negative. 

The $(y-1)$ exponent can be written as $d(A) - |A|-(v(A) - k(A)+\frac 12 \gamma(A))$. Here $d(A) - (|A|+v(A)) + k(A)$ is the nullity of the bipartite graph $G$ corresponding to $\bH_{|A}$. In any map $G$, the Euler genus $\gamma(G)$ cannot be greater than the twice the nullity. To see this  start with a spanning tree of each connected component of $G$ embedded in the sphere. The number of remaining edges is equal to the nullity and adding these edges one at a time increases $\gamma$ by at most two at each stage. Thus $d(A)-|A|-\rho(A)\geq0$. (A more sophisticated argument, for example by considering the homology generators, will  show $\gamma(G)$ cannot be greater than the the nullity.)
\end{proof}

Since when $\bH$ is orientable $\rho(A)$ is integral for each subset of edges $A$, the following holds. 

\begin{proposition}\label{prop:orienttpoly}
If $\bH$ is orientable, then $T(\bH;x,y)$ can be expanded as a polynomial in $x$ and $y$.
\end{proposition}

The converse of Proposition~\ref{prop:orienttpoly} is false, as shown by the gehm shown in Figure~\ref{f.ex4a} which is non-orientable but has Tutte polynomial $x+y-2$.

\begin{figure}[ht!]
     \centering
  %%%%%%%%%%%%%
  %%%%%2222%%%%%%%%
        \hfill
        \begin{subfigure}[c]{0.45\textwidth}
        \centering
\centering
    \includegraphics[scale=1]{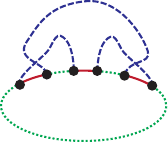}
        \caption{A nonorientable gehm.}
        \label{f.ex4a}
     \end{subfigure}
%%%%
  %%%%%1111%%%%%%%%
        \hfill
        \begin{subfigure}[c]{0.45\textwidth}
        \centering
  \includegraphics[scale=1]{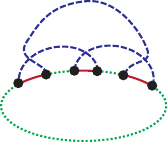}       
  \caption{An orientable gehm.}
        \label{f.ex4b}
     \end{subfigure}
\caption{Two gehms with the same Tutte polynomial  $x+y-2$.}
\label{f.ex4}
\end{figure}

By comparing the exponents of the corresponding terms in the sums expressing $T$ and $Z$ we easily obtain the following translation between the two functions.
\begin{proposition}\label{prop:tuttedichrom}
For a gehm $\bH$, 
\[ T(\bH;x+1,y+1) = {\sqrt x}^{d(\bH)-e(\bH)-f(\bH)} {\sqrt y}^{-v(\bH)} Z\Big(\bH;\frac{\sqrt y}{\sqrt x},\sqrt x \sqrt y\Big).\]
\end{proposition}

This enables us to obtain the following deletion-contraction recurrence and duality relation for $T$.
\begin{theorem}\label{thm:delcon}
    Let $\bH$ be a gehm. Then $T(\bH;x,y)=1$ if $\bH$ has no hyperedges and otherwise for each hyperedge $e$,
\begin{multline*} T(\bH;x,y) = {\sqrt {x-1}}^{\,f(\bH\ba e) - f(\bH)+d(e)-1}\, T(\bH\ba e; x,y) \\+ 
{\sqrt {y-1}}^{\,v(\bH\con e) - v(\bH)+d(e)-1} \,T(\bH\con e; x,y).
\end{multline*}
\end{theorem}

\begin{proof}
Using Proposition~\ref{prop:tuttedichrom}, Corollary~\ref{cor:dichromdelcon} and then Proposition~\ref{prop:tuttedichrom} again, we obtain
\begin{align*}
T(\bH;x+1,y+1) =& {\sqrt{x}}^{d(\bH)-e(\bH)-f(\bH)} {\sqrt y}^{-v(\bH)}
\,Z\Big(\bH;\tfrac{\sqrt y}{\sqrt x},\sqrt x \sqrt y\Big)\\
=& {\sqrt{x}}^{d(\bH)-e(\bH)-f(\bH)} {\sqrt y}^{-v(\bH)}
\Big( Z\Big(\bH\ba e; \tfrac{\sqrt y}{\sqrt x},\sqrt x \sqrt y\Big)\\ &\quad + \Big(\frac{\sqrt y}{\sqrt x}\Big)^{d(e)-1}  
\,Z\Big(\bH\con e; \tfrac{\sqrt y}{\sqrt x},\sqrt x \sqrt y\Big)\Big)\\
=& {\sqrt{x}}^{d(e)-1+f(\bH\ba e)-f(\bH)} {\sqrt y}^{v(\bH \ba e)-v(\bH)}
\,T(\bH\ba e;x+1,y+1) \\
&\quad+ {\sqrt{x}}^{f(\bH\con e)-f(\bH)} {\sqrt y}^{d(e)-1+v(\bH \con e)-v(\bH)}
\,T(\bH\con e;x+1,y+1).
\end{align*}
The result follows by noting that $v(\bH \ba e)=v(\bH)$ and dually that $f(\bH\con e)=f(\bH)$.
\end{proof}

\begin{proposition}\label{prop:tdual}
For a gehm $\bH$, 
\[ T({\bH}^*;x,y) = T(\bH;y,x).\]
\end{proposition}
\begin{proof}
By using Proposition~\ref{prop:tuttedichrom}, Corollary~\ref{cor:dichromdual} and then Proposition~\ref{prop:tuttedichrom} again, we obtain
\begin{align*} T(\bH;x+1,y+1) &= 
{\sqrt x}^{\,d(\bH)-e(\bH)-f(\bH)} {\sqrt y}^{\,-v(\bH)} Z\left(\bH;\frac{\sqrt y}{\,\sqrt x},\sqrt x \sqrt y\right)\\
&=  {\sqrt x}^{\,-f(\bH)} {\sqrt y}^{\,d(\bH)-e(\bH)-v(\bH)}   \, Z\left(\bH^*;\frac{\sqrt x}{\sqrt y},\sqrt x \sqrt y\right) \\
&={\sqrt x}^{\,-v(\bH^*)} {\sqrt y}^{\,d(\bH^*)-e(\bH^*)-f(\bH^*)}   \, Z\left(\bH^*;\frac{\sqrt x}{\sqrt y},\sqrt x \sqrt y\right) \\
&= T(\bH^*;y+1,x+1).
\end{align*}
\end{proof}
It is worth emphasising that Proposition~\ref{prop:tdual} holds for all gehms, not just those of genus 0, in contrast with the more refined hypermap polynomial of~\cite[Theorem~2.16]{CHpreWhit} which only holds for genus 0 hypermaps.

As all the relevant parameters are additive over components, if $\bH_1$ and $\bH_2$ are disjoint gehms, we have $T(\bH_1 \sqcup \bH_2;x,y) = T(\bH_1;x,y)\,T(\bH_2;x,y)$. 

Now let $\bH_1$ and $\bH_2$ be disjoint gehms such that
for $i=1,2$, $\bH_i$ includes a $g$-edge $e_i=x_iy_i$ which is not an isolate. Then the join of $\bH_1$ and $\bH_2$ along $e_1$ and $e_2$, denoted by $\bH_1 \myjoin{e_1}{e_2} \bH_2$, is obtained by forming the disjoint union of $\bH_1$ and $\bH_2$, and then replacing the $g$-edges $e_1$ and $e_2$ by $g$-edges $x_1x_2$ and $y_1y_2$.  

\begin{proposition}
Let $\bH_1$ and $\bH_2$ be disjoint gehms such that
for $i=1,2$, $\bH_i$ includes a $g$-edge $e_i=x_iy_i$ which is not an isolate. Then we have
\[ T(\bH_1 \myjoin{e_1}{e_2} \bH_2;x,y) = T(\bH_1;x,y)\,T(\bH_2;x,y).\]
\end{proposition}

\begin{proof}
    Let $\bH=\bH_1 \myjoin{e_1}{e_2} \bH_2$. Observe that 
   for any subset $A$ of the hyperedges of $\bH$, we have
    $v(\bH_{|A}) = v((\bH_1 \sqcup \bH_2)_{|A})-1$ and 
$f(\bH_{|A}) = f((\bH_1 \sqcup \bH_2)_{|A})-1$. Hence 
$\rho(\bH_{|A}) = \rho((\bH_1 \sqcup \bH_2)_{|A})$. Thus
\[ T(\bH;x,y) = T(\bH_1\sqcup \bH_2;x,y)  = T(\bH_1;x,y)\,T(\bH_2;x,y). \]
\end{proof}

In order to state some evaluations of $T$ we make the following definitions.
Notice that for a gehm $\bH$ we have $v(\bH)\leq d(\bH)-e(\bH)+k(\bH)$. (For any graph $G$ we have $e(G)\geq v(G)-k(G)$. The gehm inequality follows by applying this to the underlying bipartite graph of $\bH$  as a hypermap.) When equality holds we say that $\bH$ is a \emph{hyperforest}. Then a \emph{hypertree} is a connected hyperforest.
For example, deleting the degree four hyperedge of Figure~\ref{f.ex1b} gives a hypertree.

For a hypertree $\bH$, we have
\[ \gamma(\bH) = 2-f(\bH)+d(\bH)-e(\bH)-v(\bH) = 1-f(\bH).\]
As $\gamma(\bH)\geq 0$ and $f(\bH)\geq 1$, we deduce that $f(\bH)=1$ and $\gamma(\bH)=0$.

For a gehm $\bH$, we define $t(\bH)$ to be its number of \emph{spanning hypertrees}, that is, the number of subsets $A$ of $E(\bH)$ so that $\bH_{|A}$ is a hypertree.

\begin{proposition}\label{prop:evaluations}
Let $\bH$ be a gehm. Then 
\begin{enumerate} 
\item $T(\bH;2,2)=2^{e(\bH)}$.
\item $T$ does not detect orientability.
\item If $\bH$ is connected, then 
\[ T(\bH;1,1)= \begin{cases} 
t(\bH) & \text{if $\bH$ has Euler genus $0$,} \\
0 & \text{otherwise.}   
\end{cases} \]
\end{enumerate}
\end{proposition}

\begin{proof}
\begin{enumerate}
\item When $x=y=2$, every term in the sum in the right-side of Equation~\eqref{eq:Tuttedef} equals $1$, and the result follows easily.
\item For example, the gehms in Figure~\ref{f.ex4}  share the same Tutte polynomial, namely $x+y-2$, but only the gehm in Figure~\ref{f.ex4b} is orientable. 
\item Clearly $T(\bH;1,1)$ is equal to the number of subsets $A$ of $E(\bH)$ with $\rho(A)=\rho(\bH)=d(A)-|A|$. 
As $\bH$ is connected, we have $\rho(\bH)=v(\bH)-1+\gamma(\bH)/2$.

Next suppose that $\gamma(\bH)=0$ and that $\bH_{|A}$ is a hypertree. Then $\rho(A)=v(A)-1$ and, from the definition of a hypertree, $d(A)-|A|=v(A)-1$.
Hence $\rho(A)=\rho(\bH)=d(A)-|A|$ and $A$ contributes one to $T(\bH;1,1)$. 

Now let $A$ be a subset of $E(\bH)$ with $\rho(A)=\rho(\bH)=d(A)-|A|$.
The condition $\rho(A)=d(A)-|A|$ is equivalent to 
\begin{equation} \label{eq:r(A)=A} 
d(A)-|A| = v(A) - f(A).
\end{equation}
The non-negativity of $\gamma(A)$ gives
\begin{equation} \label{eq:genusstuff} d(A)-|A|\geq f(A)+v(A)-2k(A).\end{equation} 
By combining these two equations we get $f(A)\leq k(A)$ which gives $f(A)=k(A)$. Then Equation~\eqref{eq:r(A)=A} yields $v(A)=d(A)-|A|+k(A)$ which implies that $\bH_{|A}$ is a hyperforest. Moreover, we also get
\[ \rho(A)= d(A)-|A| = v(\bH)-k(A).\]
Thus the condition $\rho(A)=\rho(\bH)$ is equivalent to $v(\bH)-k(A) = v(\bH)-1+\gamma(\bH)/2$ which is only satisfied when $\gamma(\bH)=0$ and $k(A)=1$, that is, when $\gamma(\bH)=0$ and $\bH_{|A}$ is a hypertree.
\end{enumerate}
\end{proof}

\section{Connections with other polynomials}

\subsection{Classical and topological Tutte polynomials}\label{ss:classical}
We begin by describing the coincidence of $T(\bH;x,y)$  with the Tutte polynomials of graphs and maps.
If $\bH$ is a gem and therefore represents a graph embedded in surface then $T(\bH;x,y)$ coincides with the \emph{ribbon graph polynomial}, also known as the \emph{2-variable Bollob\'as--Riordan polynomial} or the \emph{Tutte polynomial of cellularly embedded graphs}. The ribbon graph polynomial is an important and well-studied map analogue of the Tutte polynomial~\cite{zbMATH07635160,zbMATH07395405,zbMATH06832600,iainmatrixannals}. When $\bH$ is a gem, $T(\bH;x,y)$ is also a specialisation of the  \emph{Bollob\'as--Riordan polynomial} of~\cite{zbMATH01801590}, with $ T(\bH;x,y)= (x-1)^{-\gamma(\bH)/2} R(\bH; x,y-1,\sqrt{(x-1)(y-1)},1) $.     
If the gem $\bH$ represents a graph $G$ embedded in the plane, then by Euler's Formula $\rho(A)$ equals the rank of the graph $G\ba (E-A)$. It follows that in this case $ T(\bH;x,y)=T(G;x,y)$ where $T(G;x,y)$ is the classical Tutte polynomial of the graph $G$.

\subsection{Transition polynomials}

Here we use the term \emph{Eulerian graph} to mean a graph in which each vertex is of even degree (and in particular it need not be connected). In order to accommodate isolates, we also allow Eulerian graphs to have  edges that are incident with no vertices. We call these \emph{free loops}.

The multivariable dichromatic polynomial $Z(\bH;\boldsymbol{u},v)$ can  be seen to be an evaluation of the transition polynomial of an Eulerian graph.  To make this connection, we must first define the medial map of a gehm.

Let $\bH$ be a gehm. Its \emph{medial map} $\bH_m$ is an  Eulerian map defined as follows.
Consider the natural embedding of the gehm in a surface $\Sigma$ as given in Construction~\ref{makehypermap}. Recall that each face in the embedding corresponds to a hypervertex, hyperedge or hyperface. 
Then $\bH_m$ is the graph embedded in $\Sigma$ constructed as follows. 
Retain all isolates, so each  isolate becomes a circle embedded in a sphere. 
For all non-isolate components proceed as follows.   
For vertices of  $\bH_m$, place one vertex in the interior of each of the faces  of the embedded gehm that corresponds to a hyperedge. 
Form the edges of $\bH_m$ by embedding non-intersecting arcs as follows. For each vertex $w_e$ of $\bH_m$ corresponding to a hyperedge $e$, embed non-intersecting arcs from each $w_e$ to each gehm-vertex of the hyperedge $e$. Now for each $g$-edge in the gehm, join up the two arcs meeting its end-vertices to form an edge of $\bH_m$, as in Figure~\ref{fig:med1}.

Note that each face of the map $\bH_m$, other than those in components arising from isolates, corresponds to either a hypervertex or hyperface of $\bH$. Colour the faces corresponding to hypervertices grey and those corresponding to hyperfaces white. For the components arising from isolates, assign one colour to each face. This results in a proper face 2-colouring. We call this a \emph{natural checkerboard colouring} of $\bH_m$.

\begin{figure}[ht!]
     \centering
  %%%%%%%%%%%%%
\hfill
   \begin{subfigure}[c]{0.45\textwidth}
        \centering
  \labellist
        \small\hair 2pt
        \pinlabel $S^2$ at 165 130 
      \endlabellist
      \includegraphics[scale=0.8]{f2natembed}
        \caption{Natural embedding of the gehm $\bH $ in the sphere.}
        \label{f.ex3new}
     \end{subfigure}
  %%%%%2222%%%%%%%%
        \hfill
           \begin{subfigure}[c]{0.45\textwidth}
                \centering
  \labellist
        \small\hair 2pt
        \pinlabel $S^2$ at 165 130 
      \endlabellist
      \includegraphics[scale=0.8]{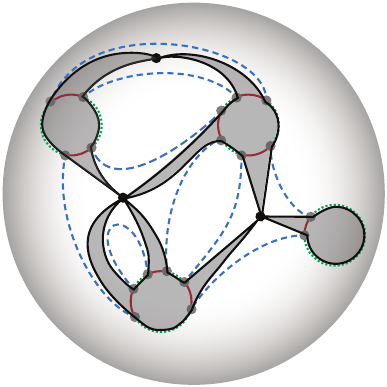}
        
        \caption{The face 2-coloured  medial map $\bH_m$ in the sphere, superimposed on the natural embedding of the gehm.}
        \label{f.HMedial}
     \end{subfigure}
\caption{Creating and face 2-colouring the medial map $\bH_m$ of a gehm $\bH$.}
\label{fig:med1}
\end{figure}

We now recall the generalized transition polynomial and define a specialisation of it that, when applied to the medial map, agrees with the dichromatic polynomial of the gehm.  The generalised transition polynomial, $q(G; W,t)$,  of \cite{E-MS02} is a multivariate graph polynomial that generalises Jaeger's transition polynomial of~\cite{Ja90}.   In \cite{EMMa}, the authors specialised the generalised transition polynomial to maps, calling this specialisation the \emph{topological transition polynomial}.  Our hypermap (or gehm) transition polynomial uses analogous ideas.

A \emph{vertex state} at a vertex $w$ of an Eulerian graph is a partition of the half-edges incident with $w$ into pairs.
The corresponding \emph{smoothing} at a vertex $w$ 
is the result of the following process for all half-edges that are paired.
If $(u,w)$ and $(v,w)$ are two non-loop edges whose half-edges are paired at the vertex $w$, then we replace these two edges with a single edge $(u,v)$.  In the case of a loop, we temporarily insert an extra vertex of degree two on the loop, carry out the operation, and then suppress the temporary vertex.

A {\em graph state} of an Eulerian graph $G$ is a choice of vertex state at each of its vertices.
A  set of free loops is obtained from a graph state $s$ by smoothing each vertex state in it. We let $k(s)$ denote the number of free loops arising from $s$ in this way.
If $G$ has no vertices, and so is either empty or a collection of free loops then its unique graph state is just itself.
We let $\mathcal{S}(G)$ denote the set of all graph states of $G$.

We can now assign weights to these vertex and graph states.  
\begin{definition}
Let $\mathcal{R}$ be a commutative ring with unity.  
\begin{itemize}

 \item A \emph{pair weight} is an association of a value $p(e_w, f_w) \in \mathcal{R}$ to a pair
of half-edges incident with a vertex $w$.  

\item A \emph{weight system} of an Eulerian graph $G$, denoted $\Omega(G)$, or simply $\Omega$ when $G$ is clear from context,  is an assignment of a pair weight to every possible pair of adjacent half-edges of $G$.

 \item
The \emph{vertex state weight} of a vertex state of a vertex $w$ is  $\prod p(e_w, f_w)$ over all pairs $(e_w, f_w)$ forming the vertex state.

\item  The \emph{state weight} of a graph state $s$ of a graph $G$ with weight
system $\Omega$ is $\omega(s) = \prod \omega(w,s)$ where $\omega(w,s)$ is the vertex state weight of the
vertex state at $w$ in the graph state $s$, and where the product is over all vertices of
$G$.
\end{itemize}
\end{definition}

Note that in many specialisations of the generalised transition polynomial, it is common to give just vertex state weights for particular vertex states.  Thus, if a vertex has degree $2n$, then implicit in giving just a vertex state weight of say $\alpha$ for some vertex state is that all the pair weights for the edge pairs comprising that state are $ \alpha^{1/n}$. This assignment of pair weights of course has to be consistent across all the vertex states. 
It is also common to use additional information, such as the cyclic order of the edges about a vertex in an embedding or a face colouring to determine vertex state weights.  

Figure~\ref{VertWeights} shows a vertex $v$ of degree $4$ within a graph $G$. So a weight system for $G$ would include $6$ pair weights corresponding to the $\binom 42$
pairs of half-edges at $v$. There are three vertex states at $v$, corresponding to the three ways of partitioning the four half-edges into sets of size two. We are usually interested in the product of the pair weights corresponding to the pairing of half-edges in a vertex state, so as noted above it is common to specify the vertex state weights, providing these are consistent. Issues of consistency never arise for a vertex of degree four, and a potential set of the three vertex state weights at $v$ is also shown in~Figure~\ref{VertWeights}. In this case $G$ is a two face-coloured   map and the weights are determined from the colouring.

\begin{figure}[ht!]
\centering   
\labellist
        \small\hair 2pt
        \pinlabel $a$ at 211 15 
        \pinlabel $b$ at   330 15 
        \pinlabel $c$ at  457 15
      \endlabellist
\includegraphics[scale=0.6]{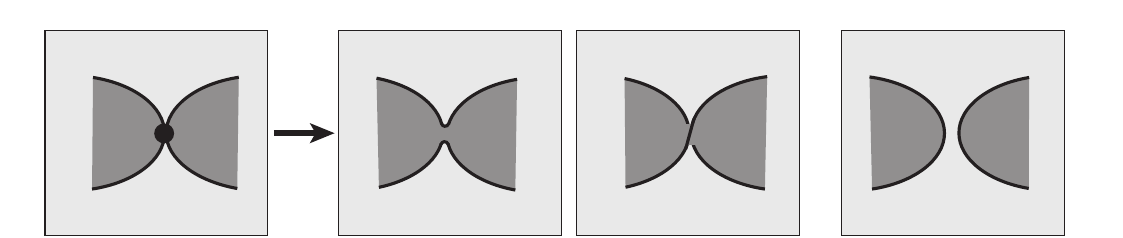}
\caption{A set of vertex state weights for all vertex states at a vertex of degree 4 in a face 2-coloured map (showing part of the surface), where $a,b,c\in\mathcal{R}$.}
\label{VertWeights}
\end{figure}

Figure~\ref{TransPolyTermA} shows a face $2$-coloured map $\bG$ with two vertices, both of degree four. (The map is in the sphere with the drawing on the page indicating the embedding.) So $\bG$ has nine graph states. One of these graph states is shown in  Figure~\ref{TransPolyTermB} with the labels on the vertices indicating the vertex state weights following the scheme given in Figure~\ref{VertWeights}. Thus if $s$ is this graph state, $\omega(s)=ab$ and $k(s)=1$.

\begin{figure}[ht!]
 \centering
  %%%%%%%%%%%%%
\hfill
   \begin{subfigure}[t]{0.45\textwidth}
        \centering
        \labellist
        \small\hair 2pt
        \pinlabel $\subset S^2$ at 275 15 
      \endlabellist
\includegraphics[scale=0.5]{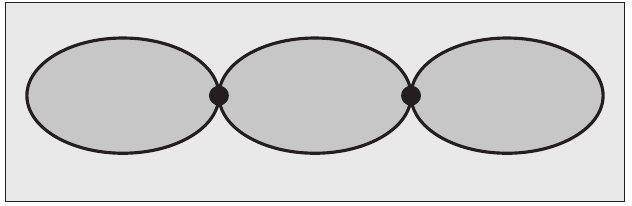}        \caption{A plane, face 2-coloured map $\bG$.}
        \label{TransPolyTermA}
     \end{subfigure}
     \hfill
        \begin{subfigure}[t]{0.45\textwidth}
        \centering
\labellist
        \small\hair 2pt
        \pinlabel $a$ at 100 15 
        \pinlabel $b$ at   192 15 
                \pinlabel $\subset S^2$ at 275 15 
      \endlabellist
\includegraphics[scale=0.5]{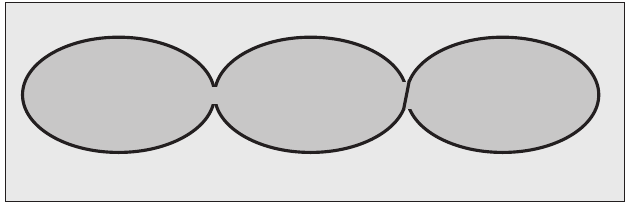}        \caption{One of the graph states of $\bG$ showing its vertex state weights.}
        \label{TransPolyTermB}
     \end{subfigure}
\caption{An example of a map and a graph state.}
\label{TransPolyTerm}
\end{figure}

The generalised transition polynomial is then defined as follows.
\begin{definition}
    Let $G$ be an Eulerian graph, with weight system $\Omega$.  Then the \emph{generalised transition polynomial} is
    \begin{equation*}
        q(G,\Omega;x) = \sum_{s\in\mathcal{S}(G)} \omega(s)\, x^{k(s)}.
    \end{equation*} 
\end{definition}

We now apply the generalized transition polynomial to obtain a hypermap transition polynomial via the medial map.
 We will see that with appropriately chosen weight systems, both the coarse  Tutte polynomial for hypermaps defined here and the topological transition polynomial of~\cite{EMMa} are specialisations of the hypermap transition polynomial. 

 \begin{definition}[The hypermap transition polynomial] \label{HMq}
   The  \emph{hypermap transition polynomial} $\Phi(\bH,  \Omega , t)$
   is 
   the specialization of the generalised transition polynomial  to medial maps of gehms (i.e., hypermaps)  given by 
\begin{equation*}
\Phi(\bH,  \Omega , t) =\sum_{s \in \mathcal{S}( \bH_m)} \omega( s )\, t^{k(s)},
\end{equation*}
where $\Omega$ is a weight system for $\bH_m$.
\end{definition}

For our purposes, we need a particular specialisation of the hypermap transition polynomial, using only two types of vertex smoothings.
If $\bH$ is a gehm and  $\bH_m$ is its naturally checkerboard coloured medial map then we may distinguish two special vertex states at $w$.
Travelling round $w$ we see half-edges and faces in the cyclic order $h_0 f_{0,g} h_1 f_{1,w} \cdots f_{2d-1,w} h_{0}$ where  $h_i$ are the half-edges, $f_{i,g}$ grey faces and $f_{i,w}$ white faces. 
The \emph{$c$-state} pairs $\{h_1,h_2\},\{h_3,h_4\}, \ldots \{h_{2d-1},h_{0}\}$;  
the \emph{$d$-state} pairs $\{h_0,h_1\},\{h_2,h_3\}, \ldots \{h_{2d-2},h_{2d-1}\}$.
Note that for  vertices of degree 2 the $c$-state and $d$-state are identical.
See  Figure~\ref{fig:vstates}.

\begin{figure}[ht!]
     \centering
  %%%%%%%%%%%%%
   \begin{subfigure}[c]{0.45\textwidth}
        \centering
\includegraphics[scale=0.83]{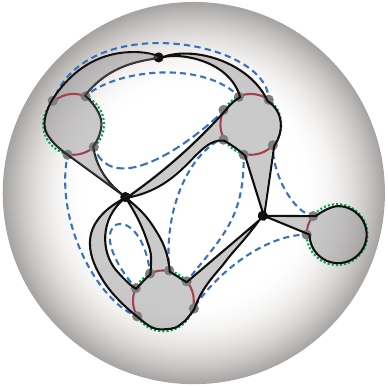}
        \caption{The medial map $\bH_m $ superimposed on the natural embedding of $\bH$ in the sphere.}
        \label{MedSmootha}
     \end{subfigure}
  %%%%%2222%%%%%%%%

\vspace{5mm}
        \begin{subfigure}[c]{0.45\textwidth}
        \centering
\centering
   \includegraphics[scale=0.83]{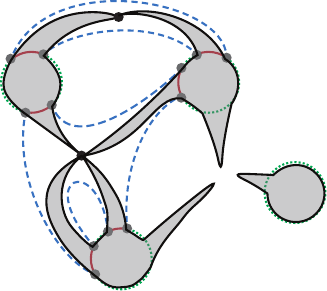}
        \caption{A $d$-smoothing in $\bH_m$ corresponding to $\bH\ba e$ where $e$ is the degree three hyperedge in $\bH$.}
        \label{MedSmoothb}
     \end{subfigure}
%%%%
  %%%%%1111%%%%%%%%
  \hfill
        \begin{subfigure}[c]{0.45\textwidth}
        \centering
\centering
\includegraphics[scale=0.83]{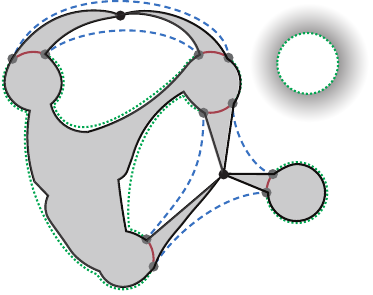}
        \caption{A $c$-smoothing in $\bH_m$ corresponding to $\bH\con f$ where $f$  is the degree four hyperedge in $\bH$. }
        \label{MedSmoothc}
     \end{subfigure}
\caption{$d$- and $c$-smoothings in a medial map corresponding to deleting and contracting a hyperedge in a gehm. In the bottom two figures, the isolate is in one sphere, while the rest of the diagram lies in a separate sphere.}
\label{fig:vstates}
\end{figure}

Using the notation in the previous paragraph, we define pair weights by setting, for vertices of degree $2d\geq 4$,  
$p(h_1,h_2)=p(h_3,h_4)= \cdots = p(h_{2d-1},h_{0})=u^{1-1/d}$, and 
$p(h_0,h_1)=p(h_2,h_3)= \cdots = p(h_{2d-2},h_{2d-1})=1$, and all other $p(h_i,h_j)=0$.
For vertices of degree two, $p(h_1,h_2)=2$. 
Let $\Omega_m(\bH)$ denote the resulting weight system.

\begin{theorem}\label{thm:tz}
Let $\bH$ be a gehm, $\bH_m$ be its medial map, and $\Omega_m(\bH)$ its medial weight system. Then.
\[  \Phi(\bH, \Omega_m(\bH), v) =   Z(\bH; u,v ) . \]
\end{theorem}
\begin{proof}
First, notice that $\omega(s) = 0$ unless all vertex states are $c$-states or $d$-states. Write $\mathcal{S}'( \bH_m)$ for the set of states consisting only of $c$-states and $d$-states.
Using $w_e$ to denote the vertex of $\bH_m$ corresponding to the hyperedge $e$ of $\bH$,
let $C(s)=\{e\in E(\bH): w_e \in V(\bH_m) \text{ has a } c\text{-state in } s\}$.  Thus, the states of $\mathcal{S}'( \bH_m)$  are in one-to-one correspondence with subsets $A \subseteq E(\bH)$ by $A_s=C(s)$. 
Second, notice that the free loops arising from the smoothing of $s$ correspond to the hyperfaces of $\bH\ba(E(\bH)-A_s)\con A_s$ which in turn correspond to the faces of $\bH_{|A_s}$.  This gives:
\begin{align*}
   \Phi(\bH, \Omega_m(\bH), v) 
   &=\sum_{s \in \mathcal{S}( \bH_m)} \omega( s )v^{k(s)}
   \\
   &=\sum_{s \in \mathcal{S'}( \bH_m)}
   \Big(\prod_{e\in C(s)} {u^{d(e)-1}}\Big)  v^{k(s)}
      \\&=\sum_{A \in E( \bH)} u^{d(A)-|A|}
      v^{f(A)}
   \\&=Z(\bH;\boldsymbol{u}, v). 
\end{align*}
\end{proof}
We note that it is straightforward to extend Theorem~\ref{thm:tz} to recover the multivariate dichromatic polynomial from the transition polynomial.

To recover the topological transition polynomial, if $\bH$ is a gem then every vertex in $\bH_m$ has degree 4. Following our previous notation we set
$p(h_1,h_2)=p(h_3,h_0)=\alpha$,
$p(h_0,h_1)=p(h_2,h_3)= \beta$, 
$p(h_0,h_2)=p(h_1,h_3)= \gamma$.
Let $\Omega_t(\bH)$ denote the resulting weight system. It is then immediate that  
  $ \Phi(\bH, \Omega_t(\bH), t) $ is the topological transition polynomial of~\cite{EMMa}.

\subsection  {Cori and Hetyei's Whitney polynomial for hypermaps}\label{CHpoly}

In~\cite{CHpreWhit}, Cori and Hetyei define a Whitney polynomial for hypermaps, $R(\bH)$, using a more refined definition of edge deletion than we use here for $Z(\bH)$. The two polynomials appear to be distinct in that it does not seem possible to recover one from the other.

We first give the edge deletion of~\cite{CHpreHyperDelCon, CHpreWhit}, which is based on hyperedge refinements, in terms of our gehm constructions, and then compare deletions and the resulting polynomials.
A refinement of a hyperedge, in terms of the gehm representation, is the result of viewing the $b$--$r$-cycle of the hyperedge as lying on a circle and replacing a subset of its $b$-edges  by the same number of non-crossing chords. 
 These new edges  are also $b$-edges and form smaller alternating $b$--$r$-cycles with the remaining $b$-edges and $r$-edges of the original hyperedge.  See Figure~\ref{Blisters}.  Any such refinement is a form of hyperedge deletion.  A \emph{total} refinement removes all the original $b$-edges from the $b$--$r$-cycle, and replaces them with chords parallel to each of the $r$-edges. These total refinements give the closest correspondence with our definition of edge deletion.  Again, see Figure~\ref{Blisters}.

 \begin{figure}[ht!]
     \centering
  %%%%%%%%%%%%%
  %%%%%2222%%%%%%%%
        %\hfill
        \begin{subfigure}[c]{0.4\textwidth}
        \centering
  \includegraphics[scale=0.9]{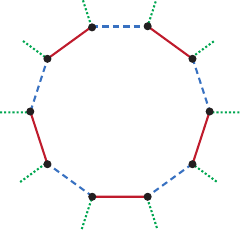}
        \caption{A  hyperedge in a gehm.}
        \label{fig:Ref1}
     \end{subfigure}
%%%%
  %%%%%1111%%%%%%%%
        \hfill
        \begin{subfigure}[c]{0.4\textwidth}
        \centering
  \includegraphics[scale=0.9]{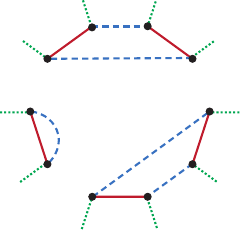}      \caption{A refinement of the hyperedge.}
        \label{fig:Ref2}
     \end{subfigure}
     %%%%%1111%%%%%%%%

\vspace{10mm}

    \begin{subfigure}[c]{0.4\textwidth}
        \centering
  \includegraphics[scale=0.9]{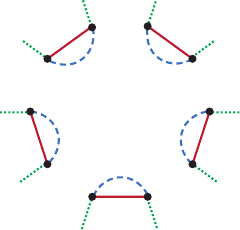}      \caption{A total refinement of the hyperedge.}
        \label{fig:Ref3}
     \end{subfigure}
     %%%%%1111%%%%%%%%
        \hfill
        \begin{subfigure}[c]{0.4\textwidth}
        \centering
  \includegraphics[scale=0.9]{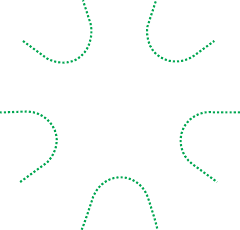} 
  \caption{Deleting the hyperedge.}
        \label{fig:Ref4}
     \end{subfigure}
\caption{Comparing edge refinements and deletion.}
\label{Blisters}
\end{figure} 

 The Whitney polynomial for hypergraphs given by Cori and Hetyei, rewritten in our framework, is 
  \[   R(\bH; u, v) = u^{-k(\bH)} v^{d(\bH)-v(\bH)} \sum_{\beta} (uv)^{k(\bH_{\beta})} v^{- e(\bH_{\beta})}.
  \]
By rephrasing the definition in this way $R(\bH; u, v)$ extends immediately to nonorientable hypermaps.
Here, the sum is over all possible hypermap refinements $\beta$, where $\beta$ gives a choice of refinement for each of the hyperedges, and $\bH_{\beta}$ is the resulting gehm. 

Due to the edges of degree one retained in the total refinements, it also does not seem possible to recover the polynomial $Z$  from the  polynomial $R$  by restricting the sum in $R$  to the total refinements. 

The difference between $R$ and $Z$  is also apparent even in the case of gems where there is a one-to-one correspondence between the refinements defining $R$ and the deletions defining $Z$.    For gems, $R$ coincides with the classical Whitney polynomial of the underlying graph, as noted in \cite{CHpreWhit}, and thus it does not retain topological information for maps.  However, $Z$ coincides with the Tutte  polynomial of maps, and thus does encode the topological information. For example, $R$ does not distinguish between two loops on a sphere and on a torus, while $Z$ does. On the other hand, the two gehms shown in Figure~\ref{fig:ZsameRnot} satisfy $Z(\bH_1;x,y)=Z(\bH_2;x,y)=v^2+u^3v^3$, but
\begin{align*}
    R(\bH_1;u,v) &= u(1+2v+v^2) +(4+5v+v^2)\\
    \intertext{and}
    R(\bH_2;u,v) &= u(1+3v+v^2) + (3+5v+v^2).    
\end{align*}

\begin{figure}[ht!]
     \centering
  %%%%%%%%%%%%%
  %%%%%2222%%%%%%%%
        \hfill
        \begin{subfigure}[c]{0.4\textwidth}
        \centering
\centering
    \includegraphics[scale=1.2]{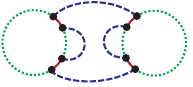}
        \caption{$\bH_1$.}
        \label{fig:ZsameRnota}
     \end{subfigure}
%%%%
  %%%%%1111%%%%%%%%
        \hfill
        \begin{subfigure}[c]{0.4\textwidth}
        \centering
  \includegraphics[scale=1.2]{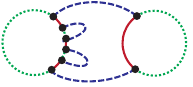}       
  \caption{$\bH_2$.}
        \label{fig:ZsameRnotb}
     \end{subfigure}
\caption{Two gehms with the same $Z$ but different $R$.}
\label{fig:ZsameRnot}
\end{figure}

\section{Concluding remarks}
There are now of course many possible directions for exploring and applying these analogues of the dichromatic, Tutte, and Whitney polynomials for hypermaps.  Among them is the question of computational complexity, and we close with a brief discussion of some complexity issues. 

In Section~\ref{ss:classical} we observed that if $\bH$ represents a graph $G$ embedded in the plane, then $T(\bH;x,y)=T(G;x,y)$. 
Vertigan proved in~\cite{zbMATH05029545} that for a fixed rational point $(x,y)$ it is $\#$P-hard to evaluate the Tutte polynomial $T(G;x,y)$ of a planar graph $G$ except when $(x-1)(y-1)\in \{1,2\}$ or $(x,y) \in \{(-1,-1),(1,1)\}$. It follows immediately that for a fixed rational point $(x,y)$ it is \mbox{$\#$P-hard} to evaluate the Tutte polynomial $T(\bH;x,y)$ of a gehm $\bH$ except possibly when $(x-1)(y-1)\in \{1,2\}$ or $(x,y) \in \{(-1,-1),(1,1)\}$. The complexity of evaluating $T(\bH;x,y)$ at the exceptional points is unclear. For example, in Proposition~\ref{prop:evaluations}, we observed that if $\bH$ is connected and has genus $0$, then $T(\bH;1,1)=t(\bH)$. Using a straightforward reduction from \textsc{Exact Cover by 3-Sets}~\cite{zbMATH03639144}, it is not difficult to show that deciding whether $t(\bH)>0$ for arbitrary hypermaps is NP-complete, but without specializing this result to hypermaps with genus zero, which seems far from straightforward, it is not possible to derive any hardness result concerning $T(\bH;1,1)$, even for arbitrary hypermaps.
Therefore we pose the following question.

\begin{question}
    Determine the complexity of computing $T(\bH;1,1)$.
\end{question}

Given that most evaluations of $T(\bH;x,y)$ are $\#$P-hard it is natural to look for appropriate parameters around which one may construct a fixed parameter tractable algorithm. Ultimately one might hope to extend Makowsky's very general approach~\cite{zbMATH02136947} for graph polynomials expressible in Monadic Second Order Logic to the hypermap setting.

\begin{acknowledgement}
Part of this work was undertaken while the authors were at the \emph{MATRIX Workshop on Uniqueness and Discernment in Graph Polynomials}.
We would like to thank MATRIX for providing a productive and inspiring environment. 
\end{acknowledgement}

\ethics{Open Access:}{For the purpose of open access, the authors have applied a Creative Commons
Attribution (CC BY) licence to any Author Accepted Manuscript version arising.}

\ethics{Data:}{No underlying data is associated with this article.}

\ethics{Competing Interests}{There are no conflicts of interest.}

\bibliographystyle{abbrv}
\bibliography{ht}

\end{document}